\documentclass[11pt,a4paper,reqno]{amsart}
\usepackage{graphics}
\usepackage{amsmath}
\usepackage{latexsym}
\usepackage{amssymb}
\usepackage[all]{xy}
\xyoption{matrix}
\xyoption{arrow}

\begin{document}

\newcommand{\lcm}{\operatorname{lcm}\nolimits}
\newcommand{\coker}{\operatorname{coker}\nolimits}
\renewcommand{\th}{\operatorname{th}\nolimits}
\newcommand{\rej}{\operatorname{rej}\nolimits}
\newcommand{\extto}{\xrightarrow}
\renewcommand{\mod}{\operatorname{mod}\nolimits}
\newcommand{\Sub}{\operatorname{Sub}\nolimits}
\newcommand{\ind}{\operatorname{ind}\nolimits}
\newcommand{\Fac}{\operatorname{Fac}\nolimits}
\newcommand{\add}{\operatorname{add}\nolimits}
\newcommand{\Hom}{\operatorname{Hom}\nolimits}
\newcommand{\Rad}{\operatorname{Rad}\nolimits}
\newcommand{\RHom}{\operatorname{RHom}\nolimits}
\newcommand{\uHom}{\operatorname{\underline{Hom}}\nolimits}
\newcommand{\End}{\operatorname{End}\nolimits}
\renewcommand{\Im}{\operatorname{Im}\nolimits}
\newcommand{\Ker}{\operatorname{Ker}\nolimits}
\newcommand{\Coker}{\operatorname{Coker}\nolimits}
\newcommand{\Ext}{\operatorname{Ext}\nolimits}
\newcommand{\op}{{\operatorname{op}}}
\newcommand{\Ab}{\operatorname{Ab}\nolimits}
\newcommand{\id}{\operatorname{id}\nolimits}
\newcommand{\pd}{\operatorname{pd}\nolimits}
\newcommand{\ql}{\operatorname{q.l.}\nolimits}
\newcommand{\rank}{\operatorname{rank}\nolimits}
\newcommand{\A}{\operatorname{\mathcal A}\nolimits}
\newcommand{\W}{\operatorname{\mathcal W}\nolimits}
\newcommand{\C}{\operatorname{\mathcal C}\nolimits}
\newcommand{\D}{\operatorname{\mathcal D}\nolimits}
\newcommand{\E}{\operatorname{\mathcal E}\nolimits}
\newcommand{\X}{\operatorname{\mathcal X}\nolimits}
\newcommand{\Y}{\operatorname{\mathcal Y}\nolimits}
\newcommand{\F}{\operatorname{\mathcal F}\nolimits}
\newcommand{\Z}{\operatorname{\mathbb Z}\nolimits}
\renewcommand{\P}{\operatorname{\mathcal P}\nolimits}
\newcommand{\T}{\operatorname{\mathcal T}\nolimits}
\newcommand{\G}{\Gamma}
\renewcommand{\L}{\Lambda}
\newcommand{\bdot}{\scriptscriptstyle\bullet}
\renewcommand{\r}{\operatorname{\underline{r}}\nolimits}
\newtheorem{lemma}{Lemma}[section]
\newtheorem{prop}[lemma]{Proposition}
\newtheorem{cor}[lemma]{Corollary}
\newtheorem{thm}[lemma]{Theorem}
\newtheorem*{thmA}{Theorem A}
\newtheorem{rem}[lemma]{Remark}
\newtheorem{defin}[lemma]{Definition}
\newtheorem{example}[lemma]{Example}

\title[Denominators]{Denominators in cluster algebras of affine type}

\author[Buan]{Aslak Bakke Buan}
\address{Institutt for matematiske fag\\
Norges teknisk-naturvitenskapelige universitet\\
N-7491 Trondheim\\
Norway}
\email{aslakb@math.ntnu.no}

\author[Marsh]{Robert J. Marsh}
\address{School of Mathematics \\
University of Leeds \\
Leeds \\
LS2 9JT \\
UK
}
\email{marsh@maths.leeds.ac.uk}


\begin{abstract}
The Fomin-Zelevinsky \emph{Laurent phenomenon} states that every cluster
variable in a cluster algebra can be expressed as a Laurent polynomial in the
variables lying in an arbitrary initial cluster.
We give representation-theoretic formulas for the denominators
of cluster variables in cluster algebras of affine type.
The formulas are in terms of the dimensions of spaces of
homomorphisms in the corresponding cluster category, and hold for any
choice of initial cluster.
\end{abstract}

\thanks{A. B. Buan was supported by a Storforsk grant no.\ 167130 from
the Norwegian Research Council. R. J. Marsh was supported by
Engineering and Physical Sciences Research Council grant EP/C01040X/2.}

\maketitle

\section*{Introduction}
Cluster algebras were introduced by Fomin and Zelevinsky in~\cite{fz02}.
They have strong links with the representation theory of finite dimensional
algebras (see e.g.\ the survey articles~\cite{bm,keller2}), with semisimple
algebraic groups and the dual semicanonical basis of a quantum group
(see e.g.\ the survey article~\cite{gls}), and with many other areas
(see e.g.\ the survey article~\cite{fz03}); these articles contain many
further references.

Here we consider acyclic coefficient-free cluster algebras of affine type,
i.e.\ those which can be given by an extended Dynkin quiver.
We give a formula expressing the denominators of cluster variables in terms 
of any given initial cluster in terms of dimensions of certain $\Hom$-spaces
in the corresponding cluster category.
The representation theory, and hence the cluster category, is well understood
in the tame case. Thus, the formula can be used to compute the
denominators explicitly.

We assume that $k$ is an algebraically closed field.
Caldero and Keller~\cite{ck2} (see also~\cite{bckmrt})
have shown, using the Caldero-Chapoton map~\cite{cc}, that for an acyclic
quiver $Q$, the cluster variables of the acyclic cluster algebra $\A_Q$ are in
bijection with the indecomposable exceptional objects in the cluster category
$\C_H$, where $H=kQ$ is the path algebra of $Q$.
Furthermore, under this correspondence the clusters correspond to
cluster-tilting objects. 
We denote by $x_M$ the cluster variable corresponding to the exceptional
indecomposable $M$ in $\C_{kQ}$.

Recall that an indecomposable regular $H$-module $X$ lies in a connected
component of the AR-quiver of $H$ known as a tube, which we denote by $\T_X$.  
For a regular indecomposable exceptional module $X$, we let $\W_X$ denote the
wing of $X$ inside $\T_X$, i.e.\ the category of subfactors of $X$ inside
$\T_X$.  We let $\tau$ denote the Auslander-Reiten translate. 

We prove the following theorem.

\begin{thmA}\label{t:main}
Let $Q$ be an extended Dynkin quiver. Let $H$ be the path algebra of $Q$,
and let $\{y_1, \dots, y_n \} = \{x_{\tau T_1}, \dots, x_{\tau T_n} \}$ 
be an arbitrary initial seed of the cluster algebra $\A_Q$, where
$T = \amalg_i T_i$ is a cluster-tilting object in $\C_{kQ}$.
Let $X$ be an exceptional object of $\C$ not isomorphic to $\tau T_i$
for any $i$. Then, in the expression $x_X = f/m$ in reduced form
we have $m = \prod_i y_i^{d_i}$, where
$$d_i = \begin{cases} \dim \Hom_{\C}(T_i,X)-1 &
\parbox[t]{6.5cm}{
if there is a tube of rank $t+1\geq 2$ containing $T_i$ and $X$,
$\ql T_i=t$ and $X \not \in \W_{\tau T_i}$,} 
\\  \dim \Hom_{\C}(T_i,X) &
\text{ otherwise.} 
\end{cases}$$
\end{thmA}

We remark that representation-theoretic expressions for denominators of
cluster variables for an arbitrary initial seed were given
in~\cite{ccs1} for type $A$ and for any simply-laced Dynkin quiver
in~\cite{ccs2,reitentodorov}.
In the general case, for an initial seed with acyclic exchange quiver,
it was shown in~\cite{bmrt,ck2} that denominators
of cluster variables are given by dimension vectors (see the next section
for more details). The general case for an arbitrary initial seed
was studied in~\cite{bmr2}. In particular, it was shown that for an affine
cluster algebra, provided the cluster-tilting object corresponding to the
initial seed contains no regular summand of maximal quasilength in its
tube, the denominators of all cluster variables are given by dimension vectors.
Cluster variables in affine cluster algebras of rank $2$ have been studied
in~\cite{cz,mp,sz,zelevinsky}.
The present article completes the denominator picture (for an arbitrary
initial seed), in terms of dimension vectors, for affine (coefficient-free)
cluster algebras.

In~\cite[6.6]{fk} it is shown that for any cluster category (and in fact in
a wider context), the dimension vector of a module coincides with the
corresponding \emph{$f$-vector} in the associated
cluster algebra with principal coefficients. See~\cite[Sec. 6]{fk} for the
definition of $f$-vectors. 

Thus our results
determine when Conjecture 7.17 of~\cite{fz07} holds for affine cluster algebras.
We also remark that in Theorem A, each exponent in the denominator is
less than or equal to the corresponding entry in the dimension vector,
in agreement with~\cite[5.8]{fk} and~\cite{dwz}.

Representation-theoretic expressions for cluster
variables have been widely studied; see for
example~\cite{ck1,dupont,hubery,palu,xx1,xx2,xx3,zhu}.
See in particular~\cite{bkl,dupont} for other aspects of cluster
combinatorics associated with tubes, and see
e.g.~\cite{musiker, parsons, propp, schiffler, st, yz} for related
combinatorial constructions.

In Section $1$, we recall some of the results described in the previous
paragraph.
In Section $2$, we recall some standard facts about tame hereditary
algebras.
In Section $3$, we study the transjective component of the cluster category, 
to prepare for the proof of Theorem A.
In Section $4$, we study regular objects in the cluster category, and
then in Section $5$ we prove the main theorem, and in Section $6$ we give
a small example to illustrate it.

\section{Preliminaries}
Let $Q$ be a finite connected acyclic quiver and $k$ an algebraically
closed field. Then $H=kQ$ denotes the (finite dimensional) path algebra
of $Q$ over $k$.
Let $D^b(H)$ be the bounded derived category of finite dimensional left
$H$-modules. The category $D^b(H)$ is a triangulated category with a
suspension functor $[1]$ (the shift). Since $H$ is hereditary, the category
$D^b(H)$ has almost split triangles; see~\cite{h}, and thus has an
autoequivalence $\tau$, the Auslander-Reiten translate. Let
$\C=\C_H=D^b(H)/\tau^{-1}[1]$ be the cluster category of $H$
(introduced in~\cite{ccs1} for type $A$ and in~\cite{bmrrt} in general).
Keller~\cite{keller1} has shown that $\C$ is triangulated.
For more information about the representation theory of finite dimensional
algebras, see~\cite{ars,ass}, and see~\cite{h} for basic properties of
derived categories.

We regard $H$-modules as objects of $\C = \C_H$
via the natural embedding of the module category of $H$ in $D^b(H)$.
For a vertex $i$ of $Q$, let $P_i$ denote the corresponding
indecomposable projective $H=kQ$-module.
Note that every indecomposable object of $\C$ is either an indecomposable
$H$-module or of the form $P_i[1]$ for some $i$.

We denote homomorphisms in $\C$ simply by $\Hom(\ , \  )$,
while $\Hom_H(\ , \ )$ denotes homomorphisms in $\mod H$ (or $D^b(H)$).
For a fixed $H$, we say that a map $X \to Y$ in $\C_H$
is an \emph{$F$-map} if it is induced by a map
$X \to \tau^{-1} Y [1]$ in $D^b(H)$, where $X,Y$ are direct sums of
$H$-modules or objects of the form $P_i[1]$.
Note that the composition of two $F$-maps is zero.
 
An $H$-module $T$ is a called a \emph{partial tilting module}
if $\Ext^1_H(T,T)=0$, an \emph{almost complete tilting module} if
in addition it has $n-1$ nonisomorphic indecomposable summands, and a
\emph{tilting module} if it has $n$ such summands (by a result of
Bongartz~\cite{bon} this is equivalent to the usual notion of a tilting
module over $H$). We shall assume throughout that all such
modules are \emph{basic}, i.e. no indecomposable summand appears with
multiplicity greater than $1$. 
For more information on tilting theory see~\cite{ahk}.

The corresponding notions of \emph{cluster-tilting object},
\emph{partial cluster-tilting object} and
\emph{almost complete cluster-tilting object} in $\C$ can be defined
similarly with reference to the property
$\Ext^1_{\C}(T,T)=0$; see~\cite{bmrrt}.
Note that every cluster-tilting object in $\C$ is induced from a tilting
module over some hereditary algebra derived equivalent
to $H$~\cite[3.3]{bmrrt}.

If $\overline{T}$ is a partial tilting module (respectively, a partial
cluster-tilting object) and $\overline{T} \amalg X$ is a tilting module
(respectively, a cluster-tilting object), then $X$ is called a
\emph{complement} of $\overline{T}$.

Let $\mathcal{A}
=\mathcal{A}(Q)\subseteq \mathbb{F}=\mathbb{Q}(x_1,x_2,\ldots ,x_n)$
be the (acyclic, coefficient-free) cluster algebra defined using the
initial seed $(\mathbf{x},Q)$, where $\mathbf{x}$ is a free generating set
$\{x_1,x_2,\ldots ,x_n\}$ for $\mathbb{F}$; see \cite{fz02}.

For an object $X$ of $\C$, let
$c_X = \prod^n_{i=1} x_i^{\dim \Hom_C(P_i,X)}$.
The following gives a connection between cluster categories and acylic
cluster algebras.

\begin{thm}\label{bmrtsecond}
(a)\ \cite[2.3]{bmrt} There is a surjective map
$$\alpha \colon \{\text{cluster variables of $\mathcal{A}(Q)$} \} \to
  \{\text{indecomposable exceptional objects in }\C\}.$$ 
It induces a surjective map 
$$\bar{\alpha} \colon \{
  \text{clusters}\} \to \{ \text{cluster-tilting objects}\}.$$ 

(b)\ \cite{ck2}
There is a bijection $\beta \colon X\to x_X$ from indecomposable
exceptional objects of $\C$ to cluster variables of $\mathcal{A}$ such that for
any indecomposable exceptional $kQ$-module $X$, we have $x_X=f/c_X$
as a quotient of integral polynomials in the $x_i$ in reduced form,
where $c_X = \prod^n_{i=1} x_i^{\dim \Hom_{\C}(P_i,X)}$. \\
(c)\ \cite{bckmrt} The maps $\alpha$ and $\beta$ are mutual inverses.
\end{thm}

We now recall some results and definitions from~\cite{bmr2}.
Assume $\Gamma$ is a quiver which is mutation-equivalent to $Q$.
By the above theorem there is a seed $(\mathbf{y},\Gamma)$ of
$\mathcal{A}$, where $\mathbf{y}=\{y_1,y_2,\ldots ,y_n\}$ is a
free generating set of $\mathbb{F}$ over $\mathbb{Q}$.
Let
$T_i=\tau^{-1}\alpha(y_i)$ for $i=1,2,\ldots ,n$, so that we have
$\alpha(y_i)=\tau T_i$.
Then $\amalg_{i=1}^n \tau T_i$ is a cluster-tilting object in $\C$ and
$\Gamma$ is the quiver of $\End_{\C}(\tau T)^{\op} \simeq \End_{\C}(T)^{\op}$
by~\cite{bmr1}.

For a polynomial $f = f(z_1, z_2 , \dots, z_n)$, we say that $f$ satisfies the
\emph{positivity condition} if $f(e_i) > 0$ for $i=1,2,\ldots ,n$, where
$e_i = (1, . . . , 1, 0, 1, . . . , 1)$ (with a $0$ in the ith position).

\begin{defin} \label{d:Tdenominator} \cite{bmr2}
Let $x$ be a cluster variable of $\mathcal{A}$ with $\alpha(x)=X$
for some exceptional indecomposable object $X$ of $\C$. We say that $x$ 
expressed in terms of the cluster $\mathbf{y}$ has a 
{\em $T$-denominator}
if either:
\begin{itemize}
\item[(I)] We have that $X$ is not isomorphic to $\tau T_i$ for any $i$, 
and $x=f/t_X$, where $f$ satisfies the positivity condition and 
$t_X=\prod_{i=1}^n y_i^{\dim \Hom_{\C}(T_i,X)}$, or 
\item[(II)] We have that $X \simeq \tau T_i$ for some $i$ and $x=y_i$. 
\end{itemize}
\end{defin}

Note that when the choice of cluster $\mathbf{y}$ is clear we also  
say that an exceptional indecomposable object $X$ of $\C$ has a
$T$-denominator, when its corresponding cluster variable has. We also
recall that in (I), the expression for $x$ must be in reduced form.
(For a contradiction, suppose that $f$ and $t_X$ have a common monomial
factor $u$. Suppose that the variable $y_i$ divides $u$.
Then $f(e_i)=0$, since $u(e_i)=0$, contradicting the fact that $f$ satisfies
the positivity condition.)

Here, in addition, we make the following definition:

\begin{defin}
Let $x$ be a cluster variable with $\alpha(x)=X$ for a
regular object $X$ of $\C$, and assume 
$x= f/\prod_{i=1}^n y_i^{d_i}$ for some cluster $\mathbf{y}$, where $f$ satisfies the positivity condition. 
We say $x$ (or $X$) has
a {\em diminished $T$-denominator} if $d_i = \dim \Hom_{\C}(T_i,X)-1$ for
all regular summands $T_i$ of $T$ with $\T_X=\T_{T_i}$ and $\ql T_i = t_i$,
where $t_i+1$ is the rank of $\T_{T_i}$,  
and $d_j = \dim \Hom_{\C}(T_j,X)$ for all other summands $T_j$.
\end{defin}

\begin{thm} \cite{bmr2}
Let $T=\amalg_{i=1}^n T_i$ be a cluster-tilting object in $\C= \C_{kQ}$ for
an acyclic quiver $Q$ and let
$\mathcal{A}=\mathcal{A}(Q)$ be the cluster algebra associated to $Q$.
Then: 
\begin{itemize}
\item[(a)] If no indecomposable direct summand of $T$ is regular then every
cluster variable of $\mathcal{A}$ has a $T$-denominator.
\item[(b)] If every cluster variable of $\mathcal{A}$ has a $T$-denominator,
then $\End_{\C}(T_i)\simeq k$ for all $i$.
\end{itemize}
Suppose in addition that $kQ$ is a tame algebra.
Then the following are equivalent: 
\begin{itemize}
\item[(i)] Every cluster variable of $\mathcal{A}$ has a $T$-denominator. 
\item[(ii)] No regular summand $T_i$ of quasi-length $t$ lies in a tube of
rank $t+1$. 
\item[(iii)] For all $i$, $\End_{\C}(T_i)\simeq k$.
\end{itemize}
\end{thm}

The main result (Theorem A) of this paper gives a precise descripton of the
denominators of all cluster variables for the tame case, i.e.\ also
including the case when $T$ has a regular summand $T_i$ of quasi-length
$t$ lying in a tube of rank $t+1$. 

Fix an almost complete (basic) cluster-tilting object $\overline{T}'$
in $\mathcal{C}$. Let $X,X^{\ast}$ be the two complements of $\overline{T}'$,
so that $T'=\overline{T}'\amalg X$ and $T''=\overline{T}'\amalg X^{\ast}$ are
cluster-tilting objects (see~\cite[5.1]{bmrrt}).
Let
\begin{equation}
X^{\ast} \overset{f}{\to} B \overset{g}{\to} X \overset{h}{\to}
\label{magic1}
\end{equation}
\begin{equation}
X \overset{f'}{\to} B' \overset{g'}{\to} X^{\ast}
\overset{h'}{\to} \label{magic2}
\end{equation}
be the exchange triangles corresponding to $X$ and $X^{\ast}$
(see~\cite[\S6]{bmrrt}), so that $B\to X$ is a minimal right
$\text{add}(\overline{T}')$-approximation of $X$ in $\mathcal{C}$ and
$B' \to X^{\ast}$ is a minimal right
$\text{add}(\overline{T}')$-approximation of $X^{\ast}$ in $\mathcal{C}$.
The following definition is crucial:

\begin{defin} \cite{bmr2} \label{d:excomp}
Let $M$ be an exceptional indecomposable object of $\C$. We say that
$M$ is \emph{compatible with an exchange pair $(X,X^{\ast})$}, if
either $X\simeq \tau M$, $X^{\ast}\simeq \tau M$, or, if neither of these
holds,
\begin{multline*}
\dim\Hom_{\C}(M,X)+\dim\Hom_{\C}(M,X^{\ast}) \\
=  \max(\dim\Hom_{\C}(M,B),\dim\Hom_{\C}(M,B')).
\end{multline*}
If $M$ is compatible with every exchange pair $(X,X^{\ast})$ in
$\C$ we call $M$ \emph{exchange compatible}.
\end{defin}

We also have:

\begin{prop} \cite{bmr2} \label{p:comp}
\begin{itemize}
\item[(a)]
Suppose that $(X,X^{\ast})$ is an exchange pair such that 
neither $X$ nor $X^{\ast}$ is isomorphic to $\tau M$.
Then the following are equivalent:
\begin{itemize}
\item[(i)] $M$ is compatible with the exchange pair $(X,X^{\ast})$.
\item[(ii)]
Either the sequence
\begin{equation}
0\to \Hom_{\C}(M,X^{\ast}) \to \Hom_{\C}(M,B) \to
\Hom_{\C}(M,X)\to 0
\label{firstsequence}
\end{equation}
is exact, or the sequence
\begin{equation}
0\to \Hom_{\C}(M,X) \to  \Hom_{\C}(M,B') \to
\Hom_{\C}(M,X^{\ast})\to 0
\label{secondsequence}
\end{equation}
is exact.
\end{itemize}
\item[(b)]Let $M$ be an exceptional indecomposable object of $\mathcal{C}$ and
suppose that $X\simeq \tau M$ or $X^{\ast}\simeq \tau M$.
Then we have that \begin{multline*}
\dim\Hom_{\C}(M,X)+\dim\Hom_{\C}(M,X^{\ast}) = \\
\max(\dim\Hom_{\C}(M,B),\dim\Hom_{\C}(M,B')) +1 \end{multline*}
\item[(c)]
Let $(\mathbf{x'},Q')$ be a seed, with $\mathbf{x'}=\{x'_1,\ldots ,x'_n\}$,
and assume that each $x'_i$ has a $T$-denominator.
Let $T'_i=\alpha(x'_i)$ for
$i=1,2,\ldots ,n$ and $T'=\amalg_{j=1}^n T'_j$.
Mutating $(\mathbf{x}',Q')$ at $x'_k$ we obtain a new cluster variable
$(x'_k)^{\ast}$.
Let $\overline{T}'=\amalg_{j\not=k}T'_j$.
and let $X^{\ast}$ be the unique indecomposable object in $\C$ with
$X^{\ast}\not\simeq T'_k = X$ such that $\overline{T}'\amalg X^{\ast}$ is a
cluster-tilting object.
Then the cluster variable $x_{X^{\ast}} = (x'_k)^{\ast}$ has a $T$-denominator if each summand $T_i$ of $T$ is compatible with the exchange pair
$(X, X^{\ast})$.
\end{itemize}
\end{prop}

Note that (c) is used as an induction step in~\cite{bmr2} for showing that
cluster variables have $T$-denominators.
Also, in~\cite{bmr2} it is shown that in (c) the cluster
variable $x_{X^{\ast}} = (x'_k)^{\ast}$ has a $T$-denominator if and
only if each summand $T_i$ of $T$ is compatible with the exchange pair
$(X, X^{\ast})$, but we shall not need this stronger statement.

\begin{prop} \cite{bmr2} \label{p:tamecompatible}
Let $H$ be a tame hereditary algebra, and let $M$ be an indecomposable
exceptional object in $\C$. Then $M$ is exchange compatible if and only if
$\End_{\C}(M)\simeq k$.
\end{prop}

\section{Tame hereditary algebras}
In this section we review some facts about tame hereditary algebras,
cluster categories and cluster algebras.

We fix a connected extended Dynkin quiver $Q$.
The category $\mod kQ$ of finite dimensional modules over the tame hereditary
algebra $H = kQ$ is well understood; see \cite{ring}.
Let $\tau$ denote the Auslander-Reiten translate. 
All indecomposable $kQ$-modules $X$ are either
preprojective, i.e.\ $\tau^m X$ is projective for some $m \geq 0$;
preinjective, i.e.\ $\tau^{-m} X$ is injective for some $m \geq 0$;
or regular, i.e.\ not preprojective or preinjective.

The Auslander-Reiten quiver of $H$ consists of:
\begin{itemize}
\item[(i)] the preprojective component, consisting exactly of the
indecomposable preprojective modules;
\item[(ii)] the preinjective component, consisting exactly of the
indecomposable preinjective modules;
\item[(iii)] a finite number $d$ of regular components called non-homogeneous
(or exceptional) tubes, $\T_1, \dots \T_d$;
\item[(iv)] an infinite set of regular components called homogeneous tubes.
\end{itemize}

For a fixed tube $\T$, there is a number $m$, such that
$\tau^m X = X$ for all indecomposable objects in $\T$. The minimal
such $m$ is the {\em rank} of $\T$. If $m=1$ then $\T$ is said to be
homogeneous.

We will also use the following facts about maps in $\mod H$.
Let $P$ (respectively, $I$ and $R$) be preprojective (respectively,
preinjective and regular) indecomposable modules,
and a $R'$ a regular indecomposable module with $\T_R \neq \T_{R'}$.
Then we have that
$\Hom_H(I,R)= \Hom_H(I,P) = \Hom_H(R,P) = \Hom_H(R,R') = 0$.

\section{The transjective component}

We will call an indecomposable object in the cluster category
{\em transjective} if it is not induced by a regular module. Note that
the transjective objects form a component of the Auslander-Reiten quiver
of $\C$.
One of our aims is to show that every transjective object has a
$T$-denominator for tame hereditary algebras.
In this section, we show that for this it is sufficient to find one
transjective cluster-tilting object all of whose summands have a
$T$-denominator.
Note that the results in this section do not require $H$ to be tame, but
hold for all finite dimensional hereditary algebras.

\begin{rem} \label{r:transjective}
We remark that, given a finite set of indecomposable transjective objects
in the cluster category, we can, by replacing the hereditary algebra $H$
with a derived equivalent hereditary algebra, assume that all of the objects
in the set are preprojective~\cite[3.3]{bmrrt}.
We shall make use of this in what follows.
\end{rem}

We start with the following observation
\begin{lemma}\label{l:AR}
Assume $(X, \tau X)$ is an exchange pair.
\begin{itemize}
\item[(a)] The AR-triangle $\tau X \to E \to X \to$ is an exchange triangle.
\item[(b)] Any exceptional object $M$ is compatible with the
exchange pair $(X,\tau X)$.    
\end{itemize}
\end{lemma}

\begin{proof}
Part (a) is well-known. By~\cite[7.5]{bmrrt}, we know that 
$\Ext^1_{\C}(X, \tau X) \simeq k$ when $(X, \tau X)$ is an exchange pair.
Hence, the AR-triangle must be isomorphic to the exchange triangle.

For part (b) we can assume that $\tau X \to E \to X \to$ is induced by an
almost split sequence in $\mod H'$ with $\C_H = \C_{H'}$, by Remark~\ref{r:transjective}.
Then we use Lemma~\cite[5.1]{bmr2} to obtain that
$\Hom_{\C}(M,\ )$ applied to the AR-triangle $\tau X \to E \to X \to$
gives an exact sequence. The claim then follows from Proposition~\ref{p:comp}.
\end{proof}

The following summarizes some facts that will be useful later.

\begin{prop}\label{old}
Let $H$ be a hereditary algebra and $U$ a tilting $H$-module. 
\begin{itemize}
\item[(a)] If $U$ is a tilting module such that $U \not \simeq H$ then there
is an indecomposable direct summand $U_i$ of $U$ which is generated by
$\bar{U}= U/U_i$. 
\item[(b)] Furthermore, if $\bar{U}= U/U_i$ generates $U_i$, and
$B \to U_i$ is the (necessarily surjective) minimal right
$\add \bar{U}$-approximation of $U_i$ and
\begin{equation} \label{e:generatingsequence}
0 \to U^{\ast} \to B \to U_i \to 0
\end{equation}
is the induced exact sequence in $\mod H$ then
the $H$-module $\bar{U} \amalg U^{\ast}$ is a tilting module in $\mod H$.
\item[(c)] The exact sequence~\eqref{e:generatingsequence}
induces an exchange triangle in $\C_H$.
\item[(d)] If $U$ is preprojective, then so is $U^{\ast}$.
\end{itemize}
\end{prop}

\begin{proof}
Part (a) is a theorem of Riedtmann and Schofield~\cite{rs}.
Part (b) is a special case of a theorem by Happel and Unger~\cite{hu}.
Part (c) is contained in~\cite{bmrrt} and part (d) is obvious. 
\end{proof}

The following is also well-known and holds for any finite dimensional
hereditary algebra $H$. Note that for $H$ of finite representation type,
all modules are by definition preprojective.

\begin{lemma}\label{l:move}
For every preprojective tilting module $U$ in $\mod H$ there is
a finite sequence of preprojective tilting modules
$$U= W_0, W_1, W_2, \dots, W_r =H,$$ and exact sequences
\begin{equation}\label{exc}
0 \to  M_j^{\ast} \to B_j \to M_j \to 0
\end{equation} 
with $B_ j\to M_j$ a minimal right $\add W_j/M_j$-approximation
of the indecomposable direct summand $M_j$ of $W_j$, and
$$W_{j+1}=(W_j/M_j) \amalg M_j^{\ast}.$$
\end{lemma}

\begin{proof}
We use the fact that the preprojective component is directed, so
there is an induced partial order on the indecomposable modules,
generated by $X \preceq Y$ if $\Hom(X,Y) \neq 0$.
For the above exchange sequences we have $M_j^{\ast} \precneqq M_j$.
The result now follows directly from Proposition~\ref{old}.
\end{proof}

Next we consider transjective exchange pairs.

\begin{lemma}\label{l:trans_comp}
Let $(X,X^{\ast})$ be an exchange pair, where both $X$ and $X^{\ast}$ are
transjective. Then any regular indecomposable exceptional $M$ is compatible
with $(X,X^{\ast})$.
\end{lemma}

\begin{proof}
We choose a hereditary algebra $H'$ derived equivalent to $H$ such 
that both $X$ and $X^{\ast}$ correspond to preprojective $H'$-modules
(see Remark~\ref{r:transjective}).
Hence one of the exchange triangles, say 
$$X^{\ast} \to B \to X \to$$
is induced by a short exact sequence, by~\cite{bmrrt}.
It is clear that the middle term $B$ is also induced by a preprojective
module. Note that we have $\C_H \simeq \C_{H'}$.

We want to show that we get a short exact sequence
\begin{equation}\label{ses}
0 \to \Hom_{\C}(M,X^{\ast}) \to \Hom_{\C}(M,B) \to \Hom_{\C}(M,X) \to 0.
\end{equation}

Since there is a path of $H'$-maps from $X^{\ast}$ to $X$ in the
preprojective component of $H'$, and this component is directed,
we have that there is no $H'$-map $X \to \tau X^{\ast}$.
Hence the nonzero map $X \to \tau X^{\ast}$ induced from the exchange
triangle is an $F' = F_{H'}$-map. Any map $M \to X$ is also
an $F'$-map, using that there are no $H'$-maps from regular objects
to preprojective objects. But any composition of
two $F'$-maps is zero. Hence every map $M \to X$ will factor through
$B \to X$, so the sequence \eqref{ses} is right exact.

Assume there is a map $M \to X^{\ast}$. Then this map must be an $F'$-map.
Assume the composition $M \to X^{\ast} \to B$ is zero, so that  
$M \to X^{\ast}$ factors through $\tau^{-1} X \to X^{\ast}$. 
$$
\xymatrix{
                  & M \ar[d] \ar[dl] &          & \\
\tau^{-1}X \ar[r] & X^{\ast} \ar[r]  & B \ar[r] & X 
}
$$
Then both maps $M \to \tau^{-1}X$ and $\tau^{-1}X \to X^{\ast}$ are $F'$-maps,
and hence the composition is zero. Hence the map $M \to X^{\ast}$ is zero,
and we have shown left-exactness of \eqref{ses}. 
This finishes the proof by Proposition \ref{p:comp}.
\end{proof}

A {\em slice} in $\mod H$ (see~\cite{ring}), is a tilting module $V$ with a
hereditary endomorphism ring. Note that $\End_{\C}(V)$ is hereditary
if and only if $\End_{H}(V)$ is hereditary by~\cite{abs1}.

\begin{lemma}\label{l:knitting}
Assume there is a slice $V = \amalg_i V_i$ such that each indecomposable
direct summand $V_i$ has a $T$-denominator. Then every transjective
indecomposable object has a $T$-denominator. 
\end{lemma}

\begin{proof}
This follows from combining Lemma~\ref{l:AR} with Proposition \ref{p:comp}. 
\end{proof}

\begin{lemma}\label{l:getting_back}
Assume there is a transjective cluster-tilting object $U = \amalg_i U_i$
such that each indecomposable direct summand $U_i$ has a
$T$-denominator. Then there is a slice $V = \amalg_i V_i$
such that each indecomposable direct summand $V_i$ has a $T$-denominator.
\end{lemma}

\begin{proof}
We choose a hereditary algebra $H'$ derived equivalent to $H$, so that 
all the $U_i$ are preprojective modules in $\mod H'$ and hence
$U$ is a preprojective tilting module in $\mod H'$
(see Remark~\ref{r:transjective})

It is clear that each $W_j$ in Lemma~\ref{l:move} is a cluster-tilting
object in $\C_H$, and that the object $H'$ forms a slice in $\C_H$.  
Also it is is clear that the short exact sequences \eqref{exc}
are exchange triangles in $\C_H = \C_{H'}$, with transjective end-terms.
So the claim follows from 
Propositions~\ref{p:comp} and~\ref{p:tamecompatible}
and Lemma~\ref{l:trans_comp}.
\end{proof}

We can now state the main result of this section.

\begin{prop}\label{p:one_enough}
Assume that there is a transjective cluster-tilting object $U = \amalg_i U_i$
such that each indecomposable direct summand $U_i$ has a $T$-denominator.
Then every transjective indecomposable object has a $T$-denominator. 
\end{prop}

\begin{proof}
This follows directly from combining Lemmas~\ref{l:knitting}
and~\ref{l:getting_back}.
\end{proof}

\section{Wings}
For this section assume that $H$ is a tame hereditary algebra.
We state some properties and results concerning regular objects in the cluster
category of $H$.

Recall that a module $M$ over an algebra $A$ is known as a \emph{brick} if
it is exceptional and $\End_A(M)=k$. In fact, it is known that if $A$ is
hereditary, every exceptional $A$-module is a brick.
We say that an object $M$ in the cluster category $\C$ is a \emph{$\C$-brick}
if $M$ is exceptional with $\End_{\C}(M)=k$.  

For an indecomposable exceptional regular module $M$, we can consider 
the full subcategory $\W_{M}$ of the abelian category $\T_M$, where the
objects are all subfactors of $M$ formed inside $\T_M$.
This is called the {\em wing} of $M$. The indecomposable objects in
$\W_{M}$ form a full subquiver of the AR-quiver shaped as a triangle
with vertices given by the unique quasi-simple with a non-zero map to $M$,
the unique quasi-simple which $M$ has a non-zero map to, and $M$ itself.    

Suppose that $\ql M = t$. 
We consider $\W_{M}$ as an abelian category equivalent to $\mod \Lambda_t$,
where $\Lambda_t$ is the hereditary algebra given as the path algebra of a
quiver of Dynkin type $A_t$, with linear orientation; see \cite{ring}.
The module $M$ is a projective and injective object in $\W_{M}$, and a
tilting object in $\W_M$ has exactly $t$ indecomposable
direct summands. 

The following lemma summarizes some well-known facts,
including the fact that there are bricks in the cluster category of $H$ which
are not $\C$-bricks.

\begin{lemma}\label{l:homs}
Let $M,N$ be regular exceptional indecomposable modules 
in a tube $\T$ of rank $t+1>1$ in $\C_H$.
\begin{itemize}
\item[(a)]~\cite{bmr2}
The object $M$ is not a $\C$-brick if and only if $\ql M = t$
\item[(b)] Any cluster-tilting object in $\C_H$ contains at most one object
from each tube which is not a $\C$-brick.
\item[(c)] If $\ql M = t$ then the following are equivalent:
  \begin{itemize}
    \item[(i)]   $\Hom_H(M,N) \neq 0$ 
    \item[(ii)]  $\dim \Hom_H(M,N) = 1$
    \item[(iii)] $\dim \Hom_{\C}(M,N) = 2$
    \item[(iv)]  $N \not \in \W_{\tau M}$  
  \end{itemize}
\end{itemize}
\end{lemma}

\begin{proof}
For (c), see~\cite{ring} for the fact that $\dim \Hom_H(M,N) \leq 1$,
and the fact that
$\Hom_H(M,N) \neq 0$ if and only if $\Hom_H(N, \tau^2 M) \neq 0$
if and only if $N \not \in \W_{\tau M}$.
We have
$$\Hom_{\C}(M,N) = \Hom_H(M,N) \amalg \Hom_{\D}(M, \tau^{-1}N[1])$$
and
$$\Hom_{\D}(M, \tau^{-1}N[1]) \simeq D \Hom_{\D}(N, \tau^2 M),$$
so the equivalence in (c) follows and (a) follows.

Part (b) is well-known and easy to see.
\end{proof}

The following is well-known by~\cite{strauss}.

\begin{lemma}\label{l:wing-tilt}
Assume that a cluster-tilting object $T$ in $\C_H$ has a regular summand $M$.
Then the summands of $T$ lying in $\W_M$ form a tilting object in $\W_M$.
\end{lemma}

We recall the notion of a Bongartz complement~\cite{bon}:

\begin{lemma}\label{l:bon-gen}
Let $N$ be a partial tilting module with no projective direct summands. 
Then there exists a complement
$E$, known as the \emph{Bongartz complement} of $N$, with the following
properties:
\begin{itemize}
\item[(a)] The module $E$ satisfies the following properties
  \begin{itemize}
  \item[(B1)] $\Ext^1_H(N,A) = 0$ implies $\Ext^1_H(E,A) = 0$ for any $A$
in $\mod H$.
  \item[(B2)] $\Hom_H(N,E) = 0$.
    \end{itemize}
\item[(b)] If a complement $E'$ of a partial tilting module $X$ satisfies
(B1) and (B2), then $E' \simeq E$, where $E$ is the Bongartz complement.  
\end{itemize}
\end{lemma}

\begin{proof}
See~\cite{h}.
\end{proof}

We are especially interested in the Bongartz complements of certain regular
modules.

\begin{lemma}\label{l:bon-reg}
Let $X=X_t$ be an exceptional regular indecomposable module with $\ql X = t$.
For $i = 1, \dots, t-1$, let $X_i$ be the regular indecomposable exceptional
module such that there is an irreducible monomorphism $X_i \to X_{i+1}$. 
Then there is a preprojective module $\widetilde{Q}$ such that:
\begin{itemize} 
\item[(a)] The Bongartz complement of $X=X_t$ is
$X_1 \amalg \dots \amalg X_{t-1} \amalg \widetilde{Q}$.
\item[(b)] The Bongartz complement of
$\widehat{X} = X_1 \amalg \dots \amalg X_{t-1} \amalg X_t$ is $\widetilde{Q}$.
\item[(c)] All partial tilting modules $Y$ such that $Y$ is a tilting object
in $\W_X$ have Bongartz complement $\widetilde{Q}$.
\end{itemize}
\end{lemma}

\begin{proof}
For (a), first note that $\Ext^1_H(X,\tau A) = 0$ while
$\Ext^1_H(A,\tau A) \neq 0$ for any indecomposable module $A$ which is
either preinjective or regular with $\T_A \neq \T_X$.
Hence by (B1) the summands in $\widetilde{Q}$ are either preprojective or
regular
and lie in $\T_X$.
The property (B2) shows that any regular summand of the Bongartz
complement $E$ of $X$ must be in $\W_{\tau X}$, by Lemma~\ref{l:homs}. 
The fact that $E$ is a complement implies that any regular summand must
be in $\W_{X'}$, where $X' \to X$ is an irreducible monomorphism,
since an object $Z$ in $\W_{\tau X} \setminus \W_{X'}$ has $\Ext(X,Z) \neq 0$.
We claim that for any indecomposable regular summand $E'$ of $E$ there is a
monomorphism $E' \to X$.
Assume $E'$ is an indecomposable regular summand 
of $E$. Then, if $E'$ is in $\W_{X'}$, but there is no monomorphism to $X$,
the module $\tau E'$ will satisfy $\Ext^1(X,\tau E')= 0$, while
$\Ext^1(E',\tau E') \neq 0$, a contradiction to (B1).
Since $X\amalg E$
is a cluster-tilting object in $\C$, it follows from Lemma~\ref{l:wing-tilt}
that all indecomposable regular objects in the tube of $X$ with monomorphisms
to $X$ are summands of $E$.

Part (b) is easily verified, noting that (B1) and (B2) are satisfied. 

For (c) we show that if a module $A$ satisfies $\Ext^1_H(Y,A)= 0$, then it
satisfies $\Ext^1_H(\widehat{X},A)= 0$. Then it follows that
$\Ext^1_H(\widetilde{Q},A)= 0$, which implies that $\widetilde{Q}$ satisfies (B1); (B2) is clearly
satisfied.

To see that $\Ext^1_H(\widehat{X},A)= 0$ we use that $\W_X$ is equivalent
to $\mod \Lambda_t$, where $\Lambda_t$ is the path algebra of the Dynkin
quiver $A_t$ with linear orientation.
Now let $Y_i$ be a direct summand in $Y$ which is generated by 
$Y/Y_i$, and consider the exact sequence $0 \to Y^{\ast} \to B \to Y \to 0$,
where $B \to Y$ is the minimal right $\add Y/Y_i$-approximation.
Then $\Ext^1_H(Y^{\ast},A)= 0$, since we have an epimorphism 
$\Ext^1_H(B,A) \to \Ext^1_H(Y^{\ast},A)$.
Iterating this sufficiently many times, which is possible by  
Lemma~\ref{l:move}, we get that $\Ext^1_H(\widehat{X},A)= 0$.
\end{proof}

\begin{lemma} \label{l:nicepreprojectivecomplement}
Let $X_1,\ldots ,X_t=X$ be as in Lemma~\ref{l:bon-reg}. Then there is preprojective
module $Q$ such that:
\begin{enumerate}
\item[(a)] All partial tilting modules $Y$ such that $Y$ is a tilting
object in $\W_X$ have complement $Q$.
\item[(b)] $Q$ generates $X$.
\end{enumerate}
\end{lemma}

\begin{proof}
Let $\widetilde{Q}$ be as in Lemma~\ref{l:bon-reg}, and let $r$ be the rank of the
tube containing $X$. 
We have that
$Q=\tau^{-kr}\widetilde{Q}$ is sincere for $k$ large enough, see \cite{aus-pla}.
Since $\tau^{r}$ preserves $\T$, (a) follows from
Lemma~\ref{l:bon-reg}(c) (but note that it may no longer be the case
that $Q$ is the Bongartz complement).
Let $Y_1=X,Y_2,\ldots ,Y_t$
be exceptional regular indecomposable modules such that there is an
irreducible epimorphism $Y_i\to Y_{i+1}$ for all $i$. By (a), $Q$ is
a complement of $Y_1\amalg \cdots \amalg Y_t$.
Hence $X=Y_1$ is a complement of
$\overline{U}=Q\amalg Y_2\amalg \cdots \amalg Y_t$.
Since $Q$ is sincere, so is the almost complete tilting module $\overline{U}$.
By~\cite{hu}, $\overline{U}$ has exactly two nonisomorphic indecomposable
complements, $Z$ and $Z'$, and there is a short exact sequence
$$0\to Z\to U'\to Z'\to 0$$
where $Z\to U'$ is a minimal left $\add(\overline{U})$-approximation in
$\mod H$. Thus exactly one of $Z,Z'$ is isomorphic to $X$; we claim it is $Z'$.
Since $Q$ is preprojective, $\Hom_A(X,Q)=0$.
Since the module maps from $X$ to the $Y_i$ all factor through $Y_{t-1}$,
the minimal left $\add(\overline{U})$-approximation of $X$ is
a non-zero map $X\to Y_{t-1}$ and is therefore not a monomorphism.
It follows that $Z'\cong X$
and thus that $\overline{U}$ generates $X$. Since $\Hom_A(Y_i,X)=0$ for
$2\leq i\leq t$ (from the structure of the tube containing $X$),
we see that $Q$ generates $X$ as required.
\end{proof}

\begin{lemma}\label{l:good}
Let $\T$ be a tube of rank $t+1$ and $M$ an exceptional object in $\T$ which
is not a $\C$-brick.
Let $X=X_s$ be an exceptional indecomposable with $\ql X = s \leq t$. 
\begin{itemize}
\item[(a)] 
There is a complement $N$ of $X$ in $\W_X$ all of whose summands 
lie in $\W_{\tau M}$. 
\item[(b)] The partial tilting module $X\amalg N$ has a preprojective
complement $Q$ which generates $X$.
\end{itemize}
\end{lemma}

\setlength{\unitlength}{0.2cm}
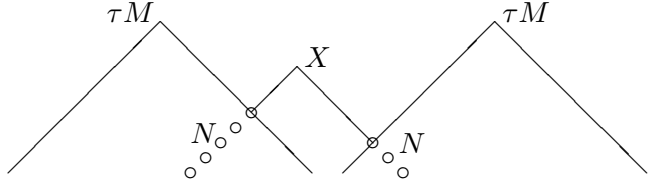
\begin{figure}
\begin{picture}(50,10)
\put(0,0){\line(1,1){10}}
\put(22,0){\line(1,1){10}}
\put(10,10){\line(1,-1){10}}
\put(32,10){\line(1,-1){10}}
\put(6.5,10){$\tau M$}
\put(32.5,10){$\tau M$}
\put(16,4){\line(1,1){3}}
\put(19,7){\line(1,-1){5}}
\put(19.5,7){$X$}
\put(12,2){$N$}
\put(25.7,1.3){$N$}
\multiput(11.5,-0.5)(1,1){5}{$\circ$}
\multiput(25.5,-0.5)(-1,1){3}{$\circ$}
\end{picture}
\caption{A complement $N$ of $X$ in $\W_X$ with summands (indicated by $\circ$) in $\W_{\tau M}$:
see Lemma~\ref{l:good}(a).}
\label{f:good}
\end{figure}

\begin{proof}
See Figure~\ref{f:good} for a pictorial representation of this lemma.
We can assume $X \not \in \W_{\tau M}$, since otherwise the result follows
directly from Lemma~\ref{l:nicepreprojectivecomplement}.

For (a) consider the relative projective tilting object in $\W_X$ given by
$X_1 \amalg \dots \amalg X_{s-1} \amalg X_s$, with $\ql X_j = j$.
If $X_{s-1} \not \in \W_{\tau M}$, consider the non-split exact sequence
$0 \to X_{s-1} \to X \to X_{s-1}' \to 0$.
We claim that $X_{s-1}'$ is in $\W_{\tau M}$.
For this, apply $\Hom_H(M,\ )$ to the above exact sequence.
Since by assumption $\Hom_H(M,X_{s-1}) \neq 0$, we have that
$\Hom_H(M,X_{s-1}) \to \Hom_H(M,X)$ is surjective. 
The map $(\Ext^1_H(M,X_{s-1}) \to \Ext^1_H(M,X)) \simeq (D\Hom_H(X_{s-1},\tau M) \to D\Hom_H(X,\tau M))$ 
is a monomorphism, since
$$\Hom_H(X,\tau M) \to \Hom_H(X_{s-1},\tau M)$$
is an epimorphism. The last statement follows since $\tau M$
is not a factor of $X_{s-1}$.
 
We also claim that $X_{s-1}'$ is a complement of
$X_1 \amalg \dots \amalg X_{s-2} \amalg X_s$ in $\W_X$.
This follows from the fact that the map $X_{s-1} \to X$ is a minimal left 
$\add X_1 \amalg \dots \amalg X_{s-2} \amalg X_s$-approximation,
together with Proposition~\ref{old}.

Now, if necessary, we exchange $X_{s-2}$ using the minimal 
$\add X_1 \amalg \dots \amalg X_{s-3} \amalg X_{s-1}' \amalg X_s$-approximation
$X_{s-2} \to X$.
The same argument as above shows that the cokernel of this map gives us a
complement in $\W_{\tau M}$.
We iterate this at most $s-1$ times, until we obtain a complement
$$N=X_1\amalg \cdots \amalg X_k\amalg X_{k+1}'\amalg \cdots \amalg X_{s-1}'$$
for $X$ in $\W_X$, with $0\leq k\leq s-1$, all of whose summands lie in
$\W_{\tau M}$, as required.

Since $X\amalg N$ is a tilting object in $\W_X$, part (b) follows
immediately from Lemma~\ref{l:nicepreprojectivecomplement}.
\end{proof}

\section{The main result}
In this section, we show the main theorem. The proof will follow from a
series of lemmas.
Throughout this section, let $T$ be a cluster-tilting object in the cluster
category $\C_H$ of a tame hereditary algebra $H$. We assume that 
$T$ has a summand which is not a $\C$-brick.
We have the following preliminary results.

\begin{lemma}\label{l:app-extend}
Let $Z$ be an exceptional indecomposable regular module.
Let $X \amalg Y$ be a tilting object in $\W_{Z}$, with $X$ indecomposable.
Assume $U \amalg X \amalg Y$ is a tilting module in $\mod H$,
where $U$ has no preinjective summands. 

\begin{itemize}
\item[(a)] Let $B \to X$ be the minimal right $\add Y$-approximation in
$\W_{Z}$ and assume there is an exchange sequence $0 \to X^{\ast} \to B \to X \to 0$ in
$\W_{Z}$. 
Then $B \to X$ is a right $\add U \amalg Y$-approximation. 
\item[(b)] Let $X \to B'$ be the minimal left $\add Y$-approximation in
$\W_{Z}$ and assume there is an exchange sequence $0 \to X \to B' \to X^{\ast} \to 0$ 
in $\W_{Z}$. Then $X \to B'$ is a left $\add U \amalg Y$-approximation. 
\end{itemize}
\end{lemma}

\begin{proof}
We prove (a), the proof of (b) is similar.
Let $U = U_p \amalg U_r$ where $U_p$ is preprojective and $U_r$
is regular. By assumption $U_r$ has no summands in $\W_{Z}$. 

We have $\Hom_{H}(U_p, B) \to \Hom_{H}(U_p, X)$ is surjective since 
$\Ext_{H}^1(U_p, X^{\ast}) \simeq D\Hom_{H}(\tau^{-1} X^{\ast},  U_p)= 0$. 

We claim that $\Hom_{H}(U_r, B) \to \Hom_{H}(U_r, X)$ is also surjective.
For this note that by the AR-structure of the tube,
there is an indecomposable direct summand $B_0$ in $B$
such that the restriction $B_0 \to X$ is surjective.
Let $U_r'$ be a summand in $U_r$ such that $\Hom_{H}(U_r', X) \neq 0$.
By assumption
$\Hom_{H}(U_r', \tau X) = 0$, since
$\Hom_{H}(U_r', \tau X) \simeq D\Ext^1(X,U_r')$.
Hence any map $U_r' \to X$ is either an epimorphism or a monomorphism. 
Since $U_r'$ is not in $\W_{Z}$, it is also not in $\W_{X}$, and it follows that any non-zero map
$U_r' \to X$ is an epimorphism, and hence
factors through $B_0 \to X$, and the claim
follows. Hence $B \to X$ is a minimal right
$\add (U \amalg Y)$-approximation. This completes the proof of (a).
\end{proof}

\begin{lemma}\label{l:tau-wing}
Let $\T$ be a tube such that $T$ has a summand $M$, lying in $\T$.
Assume $\ql M$ is maximal among the direct summands of $T$ in $\T$.
By Lemma~\ref{l:wing-tilt}, we have that
$\add \tau T \cap \W_{\tau M} = \add \tau T'$ 
for a tilting object $\tau T'$ in $\W_{\tau M}$.
Let $T= T' \amalg T''$. Then we have:
\begin{itemize}
\item[(a)] All tilting objects in $\W_{\tau M}$ are complements of $\tau T''$.
\item[(b)] All objects in $\W_{\tau M}$ have a $T$-denominator.
\end{itemize}
\end{lemma}

\begin{proof}
Note that there is a hereditary algebra $H'$, with $\C_{H'} = \C_H$, such that 
$\tau T''$ as a $H'$-module has only regular and preprojective direct summands
(see Remark~\ref{r:transjective}).
Assume $\ql M \leq t$, and that the rank of $\T$ is $t+1$.
Let $U =\tau T' = \tau N_1 \amalg \dots \amalg \tau N_{t-1} \amalg \tau M$ be
the tilting object in $\W_{\tau M}$.  

Using Proposition~\ref{old} and Lemma~\ref{l:move},
we have that all tilting objects in $\W_{\tau M}$ can be
reached from $U$ by a finite number of exchanges, given by exchange
sequences in $\W_{\tau M}$.
Using Lemma~\ref{l:app-extend} these exchange sequences are also exchange
sequences in $\mod H'$ and hence in $\C_{H'} = \C_H$. 
This shows (a).
For (b) it suffices to show that each such exchange pair is compatible with
$T$. Consider the exchange triangle
\begin{equation}\label{eqx}
X' \to \bar{X} \to X'' \to.
\end{equation}

By Proposition~\ref{p:tamecompatible}, the pair $(X',X'')$ is compatible
with all summands in $T$ which are $\C$-bricks.
It is also compatible with any regular summand $T_j$ of $T$
with $\T_{T_j} \neq \T$, since $\Hom(T_j,\ )$ vanishes on all terms of the
sequence.
By Lemma~\ref{l:homs}(a) we only need to consider compatibility with
summands of $T$ which lie in $\T$ and have quasilength $t$. By
Lemma~\ref{l:homs}(b), $T$ has at most one such summand. Since $M$ is
assumed to have maximal quasilength amongst indecomposable direct summands
of $T$ in $\T$, if $T$ has such a summand, it must be $M$. But, since the
exchange triangle~\eqref{eqx} lies inside $\W_{\tau M}$, 
we see that $\Hom(M,\ )$ vanishes when applied to \eqref{eqx}. This
finishes the proof of (b).
\end{proof}

\begin{lemma}\label{l:to-trans}
Let $X$ be an exceptional regular indecomposable object of $\C$ which is a
$\C$-brick.
\begin{itemize}
\item[(a)] An exchange pair $(X,Z)$ is compatible with any regular
object $M$ for which either $M$ is a $\C$-brick, or $\T_M \neq \T_X$, or
$X \in \W_{\tau M'}$, where $M' \to M$ is an irreducible monomorphism.

\item[(b)]  
There is an exchange triangle of the form $Y \to Q \amalg X' \to X \to$
where $X' \to X$ is an irreducible monomorphism in case $\ql X >1$ and
$X'=0$ otherwise, with the property that $Y$ and $Q$ are transjective.
\end{itemize}
\end{lemma}

\begin{proof}
(a) If $M$ is a $\C$-brick then this holds by
Proposition~\ref{p:tamecompatible}.
For the other cases note that $\Hom(M,X) = 0 = \Hom(M, \tau^{-1} X)$, and
hence when $\Hom(M, \ )$ is applied to the exchange triangle
$Z \to Q' \to X \to $, one obtains a short exact sequence.

For (b), let $X_1,\ldots ,X_t=X$ be regular exceptional indecomposable
modules such that there is an irreducible monomorphism
$X_i\to X_{i+1}$ for all $i$. Let $Q$ be the
preprojective complement of $X_1\amalg \cdots \amalg X_t$ provided by
Lemma~\ref{l:nicepreprojectivecomplement}, and let
$\overline{U}=Q\amalg X_1 \amalg \cdots \amalg X_{t-1}$.
Consider the minimal right $\add \overline{U}$-approximation
$U' \to X$ (as $H$-module). Since $Q$ generates $X$, so does
$\overline{U}$, so the approximation is surjective, and we have a short
exact sequence $0\to Y\to U' \to X\to 0$ and thus an induced approximation
triangle,
$Y \to U' \to X \to$ in $\C$.
Since all maps from $X_i$ to $X$ (with $1\leq i\leq t-1$)
factor through a non-zero map
$X_{t-1}\to X$ (taking $X_0=0$), $X'=X_{t-1}$ is the only regular
summand of $U'$ and the other summands are preprojective.
Since no non-zero map $X_{t-1}\to X$ is surjective,
while $U'\to X$ is surjective, $U'$ must have a preprojective
summand, and it follows that $Y$ is preprojective.
\end{proof}

We now deal with the transjective objects.

\begin{prop}\label{l:trans}
All transjective objects have a $T$-denominator.
\end{prop}

\begin{proof}
By Proposition~\ref{p:one_enough} it is sufficient to show that
that there is one transjective cluster-tilting object all of whose
indecomposable direct summands have $T$-denominators.
Without loss of generality we can assume that $T$ has at least one
indecomposable direct summand which is not a $\C$-brick.

Assume $T = Q \amalg R$, where $Q$ is transjective and $R$ is regular.
Then, using Lemma \ref{l:wing-tilt}, there are indecomposable summands
$M_1, \dots, M_z$ of $R$ such that
each summand of $R$ lies in one of the wings $\W_{M_i}$. We choose
a minimal such set of summands.
Since $\Ext^1_{\C}(M_i,A)\not=0$ for any object $A$ whose wing
overlaps $\W_{M_i}$, any two of the $\W_{M_i}$ must be either equal
or disjoint.

By definition, all summands of $\tau T$ have $T$-denominators.
By Lemma~\ref{l:tau-wing}, we can, for each $i$, replace the summands
of $\tau T$ in $\W_{\tau M_i}$ with the indecomposable objects in the
tube of $M_i$ which have a monomorphism to $\tau M_i$.
We obtain a new cluster-tilting
object $U=(\amalg_{i=1}^z \tau M_i) \amalg U'$ all of whose indecomposable
direct summands have $T$-denominators.

Suppose $M_1$ has quasilength $t$ and let $N_1,N_2,\ldots ,N_t=M_1$ be the
indecomposable objects in $\T_{M_1}$ with monomorphisms to $M_1$, where
$\ql(N_i)=i$ for all $i$. Then we can write $U=(\amalg_{i=1}^t \tau N_i)
\amalg Y$.
We claim that, via a sequence of exchanges, the $\tau N_i$ can be replaced
by transjective summands $Q_i$ which have $T$-denominators.
When repeating this for $M_1,M_2,\ldots ,M_z$, we will end up with a
transjective cluster-tilting object having $T$-denominators as required.

We exchange $\tau M_1$ with a complement $(\tau M_1)^{\ast}$, via the
exchange triangles:
$$(\tau M_1)^{\ast} \to B \to \tau M_1 \to$$
$$\tau M_1 \to B' \to (\tau M_1)^{\ast} \to.$$
\textbf{Claim:} The object $(\tau M_1)^{\ast}$ is transjective. \\
If $(\tau M_1)^{\ast}$ is not induced by an $H$-module, it is induced by the
shift of a projective module, and we are done.
So we can assume that $(\tau M_1)^{\ast}$ is induced by a module.
Then one of these two exchange triangles must arise from a short exact
sequence of modules.

If it is the first, then clearly $\Hom_H(X,\tau M_1)=0$ for any regular
summand $X$ of $U$ not in $\T_{M_1}$. But if $X$ lies in $\T_{M_1}$ and not in
$\W_{\tau {M_1}}$, again $\Hom_H(X,\tau M_1)=0$ since the wings
$\W_{\tau M_i}$ do not overlap (and $\ql(M_i)$ is less than the rank of its
tube for all $i$). Let $N_0=0$. Since $\tau N_{t-1}$ does not
generate $\tau N_t=\tau M_1$, it follows that $B$ has a nonzero
preprojective summand, and hence that $(\tau M_1)^{\ast}$ is preprojective.

If it is the second, then clearly $\Hom_H(\tau M_1,X)=0$ for any regular
summand $X$ of $U$ not in $\T_{M_1}$. But if $X$ lies in $\T_{M_1}$ and not in
$\W_{\tau M_1}$, again $\Hom_H(\tau M_1,X)=0$ since the wings
$\W_{\tau M_i}$ do not overlap. Since \linebreak
$\Hom_H(\tau M_1,\tau N_j)=0$ for all $j$,
it follows that $B'$ has a nonzero preinjective summand, and hence that
$(\tau M_1)^{\ast}$ is preinjective.

Hence, in either case, $(\tau M_1)^{\ast}$ is transjective. We next show that
$(\tau M_1)^{\ast}$ has a $T$-denominator, by considering two cases:

\noindent CASE I: We assume first that $\End(M_1)  = k$,
i.e.\ $M_1$ is a $\C$-brick.

Every summand of $T$ in $\T= \T_{M_1}$ is a $\C$-brick (by the choice of
the $M_i$), so by Lemma \ref{l:to-trans}(a) we obtain that the
exchange pair $(\tau M_1,(\tau M_1)^{\ast})$ is compatible with all summands of
$T$, and hence that $(\tau M_1)^{\ast}$ has a $T$-denominator by
Proposition~\ref{p:comp}. We then repeat this procedure for
$\tau N_{t-1}, \dots, \tau N_1$.
   
\noindent CASE II: $\End(M_1) \neq k$, i.e.\ $M_1$ is not a $\C$-brick.
Arguing as above, we see that we can exchange $\tau M_1$ with a transjective
object $(\tau M_1)^{\ast}$. 
$M_1$ is compatible with the exchange pair $(\tau M_1,(\tau M_1)^{\ast})$ by
definition. The other direct summands in $T$ are either
$\C$-bricks, or they are in other tubes. In both cases they are compatible with $(\tau M_1,(\tau M_1)^{\ast})$.  
Hence $T$ is compatible with that exchange pair. 
So $(\tau M_1)^{\ast}$ has a $T$-denominator by
Proposition~\ref{p:comp}. We can then exchange the other summands
$\tau N_{t-1}, \dots , \tau N_{1}$ with transjectives by Lemma \ref{l:to-trans}(b). By the last assertion of
Lemma \ref{l:to-trans}(a), each exchange pair is compatible with $M_1$.
As in case I, they are also compatible with
the other direct summands of $T$. 
Hence, we obtain a transjective cluster-tilting object having a
$T$-denominator, and we are done.
\end{proof}

\begin{lemma}
Let $\T$ be a tube such that each direct summand of $T$ lying in $\T$ is a
$\C$-brick, or such that $\T$ has no summands in $T$. Then each exceptional
indecomposable object in $\T$ has a $T$-denominator.
\end{lemma}

\begin{proof}
Let $X$ be an exceptional indecomposable object in $\T$.
We prove the Lemma by induction on the quasilength of $X$.

If $\ql X = 1$, then by Lemma~\ref{l:to-trans}(b) there is
an exchange triangle $Y \to Q  \to X \to$ with $Q$ and $Y$ transjective.
By Proposition~\ref{p:tamecompatible}, we need only show that
$(Y,X)$ is compatible with any regular non $\C$-brick
summand $M$ of $T$. But this follows from Lemma~\ref{l:to-trans}.

Now assume that any exceptional indecomposable object $Y$ of quasi-length
less than $t$ has a $T$-denominator. We want to show that the result
also holds for the exceptional indecomposable $X$ with $\ql X = t$.
For this we use Lemma \ref{l:to-trans}.
\end{proof}

It now remains to deal with the exceptional objects which are in
$\overline{\W_{\tau M}}$ for a non $\C$-brick summand $M$ of $T$. Here
$\overline{\W_{\tau M}}$ denotes the complement of the wing $\W_{\tau M}$ inside the tube $\T_M$ (i.e.\ the indecomposable objects in $\T_M$ which
are not in $\W_{\tau M}$).
For this the following lemma is crucial.

\begin{lemma}\label{l:seq}
For each indecomposable exceptional object $X^{\ast}$ in $\overline{\W_{\tau M}}$,
there are exchange triangles   
$$X^{\ast} \to B \to X \to$$
and
$$X \to B' \to X^{\ast} \to $$
such that 
\begin{itemize}
\item[(i)] \begin{multline*} \max(\dim \Hom (M,B), \dim \Hom(M,B')) \\ = \dim \Hom(M,X^{\ast}) + \dim \Hom(M,X)-1 \end{multline*} 
\item[(ii)] The object $X$ and all indecomposable summands of the
objects $B$ and $B'$ have $T$-denominators.
\item[(iii)] The object $X$ is induced by a preprojective module.
\end{itemize}
\end{lemma}

\begin{proof}
By Lemma~\ref{l:good}(a), there is an object $N$ in $\W_{X^{\ast}}$ such that
$N \amalg X^{\ast}$ is a tilting object in the wing $\W_{X^{\ast}}$ 
and all direct summands of $N$ are in $\W_{\tau M}$. 

By Lemma \ref{l:good}(b), we have that $N \amalg X^{\ast}$ has a preprojective
complement $Q$ in $\mod H$, such that $Q$ generates $X^{\ast}$.
Let $R = Q \amalg N$ and let $B' \to X^{\ast}$ (respectively, $X^{\ast} \to B$)
be the minimal right, (respectively, minimal left) $\add R$-approximations of
$X^{\ast}$.
We claim that the induced exchange triangles satisfy (i), (ii) and (iii).

Consider the exchange triangle 
$$X \to B' \to X^{\ast} \to.$$
Since $Q$ generates $X^{\ast}$ in $\mod H$, it is clear that this triangle is
induced by a short exact sequence in $\mod H$, and hence $X$ is
induced by a preprojective module (showing (iii)), since
$X\to B'$ is nonzero and $B'$ must have a preprojective
summand as $N$ doesn't generate $X^{\ast}$.

Apply $\Hom(M,\ )$ to obtain the long exact sequence
$$(M,\tau^{-1} X^{\ast}) \to (M,X) \to (M,B') \to (M,X^{\ast}) \to
(M,\tau X).$$

We claim that every $H$-map $\tau M \to X^{\ast}$ factors through $B'$. 
This follows from the configuration of $M$,$N$ and $X^{\ast}$ in the
Auslander-Reiten quiver of the tube, noting
that $\Hom_H(\tau M, N) \neq 0$ if and only if $\Hom(\tau M, X^{\ast}) \neq 0$
(if and only if $N\neq 0$).
Figure~\ref{f:seq} displays this case, when $N$ has a summand
$B'_0$ (occurring in $B'$) with $\Hom_H(\tau M,B'_0)\not=0$; compare with
Figure~\ref{f:good}.
By~\cite[Lemma 5.1]{bmr2} it follows that $(M,X) \to (M,B')$ is a
monomorphism.

\setlength{\unitlength}{0.2cm}
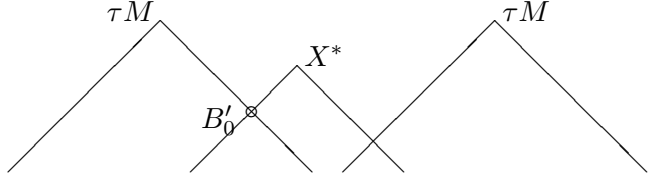
\begin{figure}
\begin{picture}(50,10)
\put(0,0){\line(1,1){10}}
\put(22,0){\line(1,1){10}}
\put(10,10){\line(1,-1){10}}
\put(32,10){\line(1,-1){10}}
\put(6.5,10){$\tau M$}
\put(32.5,10){$\tau M$}
\put(12,0){\line(1,1){7}}
\put(19,7){\line(1,-1){7}}
\put(19.5,7){$X^{\ast}$}
\put(12.7,2.9){$B'_0$}
\put(15.5,3.5){$\circ$}
\end{picture}
\caption{A summand $B_0$ of $N$ with $\Hom_H(\tau M,B_0)\neq 0$:
see proof of Lemma~\ref{l:seq}.}
\label{f:seq}
\end{figure}

We claim that $\dim \coker ((M,B') \to (M,X^{\ast})) =1$.
By Lemma \ref{l:homs}, we have that $\dim \Hom_H(M,X^{\ast}) = 1$
and it is clear that an $H$-map $M \to X^{\ast}$ will not factor through $B'$,
since $N$ is in $\W_{\tau M}$, and hence $\Hom_H(M,N) = 0$,
by Lemma~\ref{l:homs}. 

By Lemma~\ref{l:homs} the space of $F$-maps $M \to X^{\ast}$ is also
one-dimensional.
We claim that such $F$-maps will factor through $B'$. For this 
we consider two possible cases: the object $X$ is either induced by
a projective $H$-module $P$ or not. 
First assume that $X$ is non-projective. 
Since the composition of two $F$-maps is 0, it is clear that all $F$-maps
$M \to X^{\ast}$ will factor through $B' \to X^{\ast}$.
Hence the claim follows in this case.
Now consider the case where $X^{\ast}$ is projective.
Then the composition $M \to \tau^{-1}X^{\ast}[1] \to \tau^{-1}(P_i[1])[1]$ is
clearly zero, so the claim follows in this case.

We next want to show that when $\Hom(M,\ )$ is applied to the
second exchange triangle
$$X^{\ast} \to B \to X \to ,$$
we do not obtain an exact sequence. 
The map $X^{\ast} \to B$ decomposes into $X^{\ast} \to Q_0 \amalg N_0$, with
$Q_0$ preprojective and $N_0$ in $\W_{\tau M}$.
Hence $X^{\ast} \to Q_0$ is an $F$-map.
There is a non-zero $F$-map $M \to X^{\ast}$ and the composition
$M \to X^{\ast} \to B$ will be zero since $M \to X^{\ast} \to Q_0$ is the
composition of two $F$-maps and $\Hom(M, N_0)= 0$, since $M \in \W_{\tau M}$.

Hence we obtain (i), and (ii) follows using Lemmas~\ref{l:knitting}
and~\ref{l:tau-wing}(b), using the fact that $X$ and all
indecomposable summands of $B$ and $B'$ are either transjective or in
$\W_{\tau M}$, and noting that, by
Lemma~\ref{l:homs}(a), $M$ has maximal quasilength amongst the direct
summands of $T$ in $\T_M$.
\end{proof}

The proof of the following is an adaptation of parts the proof
of~\cite[Prop. 3.1]{bmrt}. It completes the proof of our main result,
Theorem A.

\begin{prop}
Let $\T$ be a tube such that $T$ has a non $\C$-brick summand $M$, lying in
$\T$. Then each object in $\overline{\W_{\tau M}}$ has a diminished
$T$-denominator.
\end{prop}

\begin{proof}
Let $X^*$ be an indecomposable object in $\overline{\W_{\tau M}}$.
By Lemma~\ref{l:seq} there is an indecomposable object $X$ and exchange
triangles
$$X^{\ast} \to B \to X \to $$
and
$$ X \to B' \to X^{\ast} \to $$
such that Lemma~\ref{l:seq} holds.

We have~\cite{bmr1} that
\begin{equation}
x_Xx_{X^{\ast}}=x_B+x_{B'}. \label{mmstar}
\end{equation}
Assume $M=T_l$.
We need to discuss two different cases.

\noindent CASE I:
Suppose that neither $X$ nor $X^{\ast}$ is isomorphic to $\tau T_i$ for
any $i$. 
Let $B=B_0\amalg B_1$, where no summand of $B_0$ is of the form
$\tau T_i$ for any $i$, and $B_1$ is in $\add \tau T$.
Similarly, write $B'=B'_0\amalg B'_1$. Let $t_{B_0}$ denote $\prod t_Y$ where
the product is over all indecomposable direct summands $Y$ of $B_0$; similarly
write $t_{B'_0}$. Note that $t_Y$ is defined in
Definition~\ref{d:Tdenominator}.

We then have $x_B=\frac{f_{B_0}y_{B_1}}{t_{B_0}}$ and
$x_{B'}=\frac{f_{B_0'}y_{B'_1}}{t_{B'_0}}$.

Let $m=\frac{\lcm(t_{B_0},t_{B'_0})}{t_{B_0}}$ and
$m'=\frac{\lcm(t_{B_0},t_{B'_0})}{t_{B'_0}}$. We then have:
$$
(x_{X^{\ast}}) = \frac{x_B + x_{B'}}{x_X} =  
\frac{(f_{B_0} m y_{B_1} + f_{B_0'} m' y_{B'_1})/f_X}{\lcm(t_B,t_{B'})/t_X},
$$
using that $t_B = t_{B_0}$ since $\Hom_{\C}(T_i, \tau T_j) = 0$ for all $i,j$, and
similarly $t_{B'} = t_{B_0'}$.

Since $M = T_l$ is a summand in $T$, we have by Lemma~\ref{l:seq} that
\begin{multline*} \max(\dim \Hom (T_l,B), \dim \Hom(T_l,B')) \\
= \dim \Hom(T_l,X^{\ast}) + \dim \Hom(T_l,X)-1. \end{multline*} 
For any other summand of $T$, say $T_i$ with $i \neq l$, we have that
$T_i$ is compatible with $(X,X^{\ast})$, and hence
\begin{multline*} \max(\dim \Hom (T_i,B), \dim \Hom(T_i,B')) \\
= \dim \Hom(T_i,X^{\ast}) + \dim \Hom(T_i,X). \end{multline*} 
We thus obtain:
\begin{eqnarray*}
t_{X}t_{X^{\ast}} & = & \prod y_i^{\dim \Hom_{\C}(T_i,X) + \dim \Hom_{\C}(T_i,X^{\ast})}   \\
 & = & y_l \cdot  \prod y_i^{\max(\dim\Hom_{\C}(T_i,B),\dim\Hom_{\C}(T_i,B'))}
     = y_l \cdot \lcm(t_B,t_{B'}).
\end{eqnarray*}
Hence
\begin{equation}
x_{X^{\ast}}=\frac{(f_{B_0} m y_{B_1} +f_{B_0'} m' y_{B'_1})/f_X}
{t_{X^{\ast}}/ y_l}.
\label{mstar}
\end{equation}
We have that $m$ and $m'$ are coprime, by definition of least common multiple. 
Since $B$ and $B'$ have no common direct
summands~\cite[6.1]{bmr1}, $y_{B_1}$ and $y_{B'_1}$ are coprime.
Suppose that $m$ and $y_{B'_1}$ had a common factor $y_i$.
Then we would have a summand $Z$
of $B'_0$ such that $\Hom_{\C}(T_i,Z) \neq 0$, and $\tau T_i$ was a
summand of $B'$. But then
$$\Ext^1_{\C}(Z,\tau T_i) \simeq D\Hom_{\C}(\tau T_i,\tau Z)
\simeq  D\Hom_{\C}(T_i,Z) \neq 0.$$
This contradicts the fact that $B'$ is the direct sum of summands of a
cluster-tilting object. Therefore $m$ and $y_{B'_1}$ are coprime,
and similarly $m'$ and $y_{B_1}$ are coprime. It follows
that $my_{B_1}$ and $m' y_{B'_1}$ are coprime.

Since all indecomposable summands of $B$ and $B'$ have $T$-denominators,
it follows (see Definition~\ref{d:Tdenominator}) that
$f_{B_0}(e_i)>0$ and $f_{B_0'}(e_i)>0$ for each $i\in \{1,2,\ldots ,n\}$.
It is clear that $(my_{B_1})(e_i)\geq 0$ and $(m'y_{B'_1})(e_i)\geq 0$.
Using that $my_{B_1}$ and $m'y_{B'_1}$ are coprime, it follows that these
two numbers cannot simultaneously be zero, so
$(f_B m y_{B_1} + f_{B'} m' y_{B'_1})(e_i)>0$.
Hence $f_Bmy_{B_1}+f_{B'}m'y_{B'_1}$ satisfies the positivity condition.
By assumption, $f_X$ also satisfies the positivity condition.

By the Laurent phenomenon~\cite[3.1]{fz02},
$x_{X^{\ast}}$ is a Laurent polynomial in $y_1,y_2,\ldots ,y_n$.
Clearly $t_{X^{\ast}}/y_l$ is also a Laurent polynomial.
Hence  
$u =(f_B my_{B_1} + f_{B'} m' y_{B'_1})/f_X =
\frac{x_{X^{\ast}}t_{X^{\ast}}}{y_l}$
is also a Laurent polynomial.
Since $u$ is defined at $e_i$ for all $i$, it must be a polynomial.
By the above, $u$ satisfies the positivity condition.

We have that $y_l$ divides
$t_{X^{\ast}} = \prod y_i^{\dim \Hom_{\C}(T_i,X^{\ast})}$, since
$\dim \Hom_{\C}(T_l,X^{\ast}) = 2$.
Hence we get that $t_{X^{\ast}}/y_l$ is a monomial.
This finishes the proof in Case (I).

\noindent CASE II:
Assume that $X \simeq \tau T_i$ for some $i$.
Note that $i \neq l$, since $X$ and hence $T_i$ is transjective,
while $T_l$ is regular.

Since $\Ext^1_{\C}(T_r,T_s) = 0$ for all $r,s$, we have that
$X^{\ast} \not \simeq \tau T_j$ for any $j$.

Using Proposition~\ref{p:comp} and Lemma~\ref{l:seq}, we have 
\begin{multline*}
\dim\Hom_{\C}(T_j,X)+\dim\Hom_{\C}(T_j,X^{\ast}) \\
= \max(\dim\Hom_{\C}(T_j,B),\dim\Hom_{\C}(T_j,B'))
+ \epsilon_j,
\end{multline*}
where
$$\epsilon_j = \begin{cases} 1 & \text{if } j=i \text{ or } j=l \\
0 & \text{otherwise} \end{cases}.$$ 

As in Case (I), but using that $x_X = y_i$ (as $X= \tau T_i$), we obtain
the expression
\begin{equation*}
x_{X^{\ast}}=\frac{(f_{B_0} m y_{B_1} +f_{B_0'} m' y_{B'_1})}{\lcm(t_B, t_{B'}) y_i}.
\end{equation*}
Using $\lcm(t_B, t_{B'}) = t_X t_{X^{\ast}} y_i^{-1}  y_l^{-1}$, we get
\begin{equation*}
x_{X^{\ast}}=\frac{(f_{B_0} m y_{B_1} +f_{B_0'} m' y_{B'_1})}
{t_{X^{\ast}}/y_l}.
\end{equation*}
As in Case (I), we get that the numerator satisfies positivity and is a
polynomial, and that $t_{X^{\ast}} y_l^{-1}$ is a monomial. The proof
is complete.
\end{proof}

\section{An example}

We give a small example illustrating the main theorem.

Let $Q$ be the extended Dynkin quiver
$$
\xymatrix{
2 \ar[r] & 3 \ar [d] \\
1 \ar[u] \ar[r] & 4 
}
$$
and let $H= kQ$ be the path algebra.
Then $H$ is a tame hereditary algebra where the AR-quiver has one exceptional tube $\mathcal{T}$, which is of rank 3.
The (exceptional part of) the AR-quiver of $\T$ is as follows, where the   
composition factors (in radical layers) of indecomposable modules are given.

$$
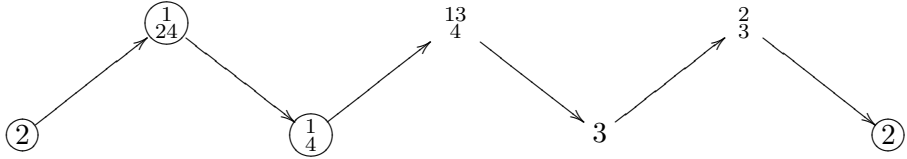
\xymatrix@!C=30pt{
& *+[o][F-]{{\begin{smallmatrix} 1 \\ 2 4 \end{smallmatrix}}} \ar[rd]& & 
{\begin{smallmatrix}1 3 \\ 4 \end{smallmatrix}} \ar[rd]& & {\begin{smallmatrix} 2 \\ 3 \end{smallmatrix}} \ar[rd] &  \\
*+[o][F-]{2} \ar[ru] & & *+[o][F-]{{\begin{smallmatrix} 1 \\ 4 \end{smallmatrix}}} \ar[ru] & & 3 \ar[ru] & & *+[o][F-]{2}  
}
$$

Let $P_i= H e_i$ denote the indecomposable projective $H$-modules.
Let $T = T_1 \amalg T_2 \amalg T_3 \amalg T_4 = 
\tau^{-1} P_4  \amalg \tau^{-1} P_1 \amalg 3 \amalg \begin{smallmatrix} 1 3 \\ 4 \end{smallmatrix}$.
It is easily verified that this is a cluster-tilting object.
The encircled modules in the above figure are those which are in $\W_{\tau T_4}$.

For each exceptional object $Y$ in the tube $\T$, we give
the dimension vector of $\Hom_{\C}(T,Y)$ as a $\End_{\C}(T)^{\op}$-module.
Note that $\Hom_{\C}(T,\tau T_3) =0= \Hom_{\C}(T,\tau T_4)$.
$$
\xymatrix@!C=30pt{
& \ast \ar@{.}[rd]& & 1102 \ar[rd] & & 1122 \ar[rd] &  \\
 0010 \ar@{.}[ru] & & \ast \ar@{.}[ru] & & 1112 \ar[ru]& & 0010 
}
$$

We consider the initial seed $\{x_1, \dots, x_4 \}$ where $x_i = x_{\tau T_i}$.
We give the corresponding cluster variables $x_Y$ (with most of the numerators skipped).
$$
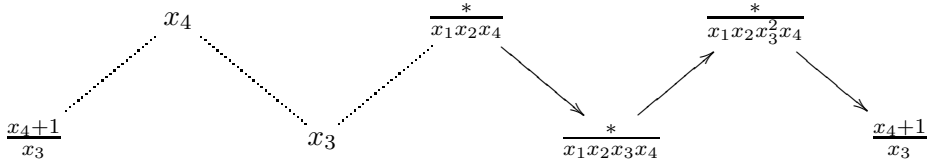
\xymatrix@!C=30pt{
 & x_4 \ar@{.}[rd]& & \frac{\ast}{x_1 x_2 x_4} \ar[rd]& & \frac{\ast}{x_1 x_2 x_3^2 x_4} \ar[rd]& & \\
  \frac{x_4+1}{x_3} \ar@{.}[ru] & & x_3 \ar@{.}[ru] & & \frac{\ast}{x_1 x_2 x_3 x_4} \ar[ru] & & \frac{x_4+1}{x_3}  
}
$$
We observe that the denominators of these cluster variables can be computed from the dimension vectors of the corresponding modules
over $\End_{\C}(T)$, as described by our main Theorem.

\vskip 0.3cm
\noindent \textbf{Acknowledgements:}
We would like to thank the referee for their very helpful comments which
allowed us to improve the presentation of and to correct an earlier version
of the article.
The first named author visited Leeds several times in the academic year
2007-2008 and he would like to thank the School of Mathematics at the
University of Leeds, and in particular Robert Marsh, for their tremendous
hospitality.

\end{document}